\documentclass[12pt,reqno]{amsart}
\usepackage[margin=0.96in]{geometry}
\usepackage{amsmath,amssymb,amsthm,graphicx,amsxtra, setspace}
\usepackage[utf8]{inputenc}
\usepackage{mathrsfs}
\usepackage{upgreek}
\usepackage{alltt}
\usepackage{relsize}
\usepackage{hyperref}
\usepackage{aliascnt}
\usepackage{tikz}
\usepackage{multicol}
\usepackage{graphicx,type1cm,eso-pic,color}
\allowdisplaybreaks
\usepackage{setspace}

\colorlet{darkblue}{blue!50!black}

\hypersetup{
	colorlinks,%
	citecolor=blue,%
	filecolor=red,%
	linkcolor=red,%
	urlcolor=blue,%
	pdfnewwindow=true,%
	pdfstartview={FitH}
}

\newtheorem{theorem}{Theorem}[section]

\newtheorem{remark}[theorem]{Remark}

\let\originalleft\left
\let\originalright\right
\renewcommand{\left}{\mathopen{}\mathclose\bgroup\originalleft}
\renewcommand{\right}{\aftergroup\egroup\originalright}

\def\C{\mathrm{C}}

\def\E{\mathbb{E}}

\def\H{\mathbb{H}}

\newcommand{\Addresses}{{
		\footnote{
			
			\noindent \textsuperscript{1}Department of Mathematics, Indian Institute of Technology Roorkee-IIT Roorkee,
			Haridwar Highway, Roorkee, Uttarakhand 247667, India.\par\nopagebreak
			\noindent  \textit{e-mail:} \texttt{Manil T. Mohan: maniltmohan@ma.iitr.ac.in, maniltmohan@gmail.com.}
			
			\textsuperscript{2}Department of Mathematics, Periyar University, Salem 636 011, Tamil Nadu, India.\par\nopagebreak
			\noindent  \textit{e-mail:} \texttt{S. Pardeep: pradeepkumar01798@gmail.com}
			\noindent  \textit{e-mail:} \texttt{S. Sankar: subusankar27@gmail.com}
				\noindent  \textit{e-mail:} \texttt{S. Karthikeyan: karthi@periyauniversity.ac.in, skkmathpu@gmail.com}
			
			\noindent \textsuperscript{*}Corresponding author.
			
			\textit{Key words:} Semilinear SPDE; L\'evy noise; Blow-up; Concavity method.

			Mathematics Subject Classification (2020): Primary: 60H15,  Secondary:  35R60;   35B44; 
			
}}}

\begin{document}
	
	
	\title[Blow-up of stochastic  semilinear parabolic equations]{Blow-up of  stochastic semilinear  parabolic equations driven by L\'evy noise
	\Addresses}
	\author[M. T. Mohan, S. Pradeep, S. Sankar and S. Karthikeyan]	{Manil T. Mohan\textsuperscript{1*}, S. Pradeep\textsuperscript{2}, S. Sankar\textsuperscript{2} and S. Karthikeyan\textsuperscript{2}}

	\maketitle
	
\begin{abstract}
The blow-up phenomena of stochastic  semilinear parabolic equations with additive as well as linear multiplicative L\'evy noises are investigated  in this work. By suitably modifying the concavity method in the stochastic context, we establish the blow-up phenomena of such systems defined on bounded domains.
\end{abstract}
	
\section{Introduction}\label{sec1}\setcounter{equation}{0}
In the past few years, researchers from a wide range of scientific fields, including theoretical physics, fluid mechanics, geophysics, pure and applied mathematics, and others, have paid significant attention to stochastic partial differential equations (SPDEs) because of their ability to model complex phenomena driven by random fluctuations.  A lot of interesting challenges, such as the well-posed problems, blow-up phenomena, asymptotic behaviors like ergodicity, random attractor, stability, etc. have been examined for several types of SPDEs. Finite-time blow-up is an important phenomenon that can occur in mathematical models and it  refers to a situation where a solution to a differential equation grows  infinite within a finite interval of time. As this behavior is often unexpected and inevitable, understanding finite-time blow-up is crucial for analyzing the behavior of mathematical models and forecasting their long-term dynamics. 
 
 The finite-time blow-up phenomenon in deterministic PDEs has been extensively studied by a good number of authors (cf. \cite{Hu2018}). The methods used to obtain  the finite-time blow-up phenomenon in deterministic PDEs defined on bounded domains  include Kaplan’s first eigenvalue method, concavity method and comparison method, which are detailed in \cite[Chapter 5]{Hu2018}. As far as the stochastic counterpart is concerned, there are only a limited results available on the existence of blow-up solutions.  In SPDEs, the system is completely determined by the initial data and  the effect of noises on the systems. Consider the following SPDEs:
   \begin{align}\label{b1}
   \left\{
   \begin{aligned}
   du(t,x) &= (\Delta u + f (u))dt + \sigma(u)dW(t), t>0,\ x\in \mathcal{O}, \\
   u(x, 0) &= u_{0} (x) \geq 0,\ x \in \mathcal{O}, \\
   u(x,t)&=0,\ t>0,\ x\in \partial \mathcal{O},
   \end{aligned}
   \right.
   \end{align}
   where $\mathcal{O}\subset\mathbb{R}^d$ is a bounded domain with smooth boundary $\partial\mathcal{O}$ and $W(\cdot)$ is a Brownian motion. Under suitable assumptions of $f(\cdot)$ and $\sigma(\cdot)$, the existence of strong solutions (in the probabilistic sense) of \eqref{b1} is studied in \cite{PLC}. Under the monotonicity and local Lipschitz assumptions on $f$, Da Prato and Zabczyk \cite{DA1992} proved that the solution of the equation \eqref{b1} exists globally with additive infinite dimensional noise (here $\sigma$ is constant). Manthey and Zausinger \cite{MZ1999} discussed the existence of a unique solution of the equation \eqref{b1}, with $\sigma$ satisfying a global Lipschitz condition on bounded domain. Dozzi and L\'opez-Mimbela \cite{Doz2010} studied the finite-time blow-up solution of the equation \eqref{b1} with $\sigma(u) = u$, $f (u) \geq u^{1+\alpha},\ \alpha > 0$ and large enough initial data, and if $f (u) \leq u^{1+\beta}$, $\beta>0$ is a constant and the initial data is small enough, then they provided a sufficient condition for the global existence of the solution to the equation \eqref{b1}. Niu and Xin \cite{Niu2012} was interested to investigate whether the explosion time prolong the life span of an explosive nonlinear SPDE by adding another independent noise to the equation \eqref{b1}.  Chow \cite{Chow2009,Chow2011} considered the blow-up phenomenon under the condition that $\sigma(u)$ does not satisfy the global Lipschitz condition.  Lv and Duan \cite{Lv2015} described the effects of the interplay between $f$ and $\sigma$ on the finite-time blow-up of the equation \eqref{b1}. 
  
 However, there are limited results available in the literature for the blow-up of solution of SPDEs driven by L\'evy noise.   By using a Lyapunov technique,  Xing and Li \cite{xing}  investigated the explosive solutions for a class of stochastic differential equations driven by Lévy processes.   Bao and Yuan \cite{Bao2016} and Li et al. \cite{Li2017} obtained the existence of local solutions of the equation of the form \eqref{b1} with jump process and L\'evy process, respectively. Li \cite{li2018} extended the above results for a  class of semilinear stochastic delayed reaction–diffusion equations with Lévy noise. A recent paper by Liang et al., \cite{liang} investigated the explosive solution of non-local SPDEs driven by Lévy noise.  
   
In the literature, the method used to prove the finite time blow-up of SPDEs on bounded domains is stochastic Kaplan's first eigenvalue method (\cite{Bao2016,PLC,Chow2011,Lv2015}).    Recently, Lv and Wei \cite{GLJW}  studied the  blow-up phenomenon of a nonlinear parabolic SPDE of the following form using the stochastic concavity method:
  \begin{align}\label{1p2}
  	\left\{
  	\begin{aligned}
  		du(x,t)&=[\Delta u(x,t)+|u(x,t)|^{p-1}u(x,t)]dt+\sigma(x,t)dW(t),\ x\in \mathcal{O},\ t>0,\\
  		u(x,t)&=0, \ x\in\partial \mathcal{O}, t\geq 0,\\ 
  		u(x,0)&=u_0(x),\ x\in \mathcal{O},
  	\end{aligned}
  	\right.
  \end{align}
  where $W(\cdot)$ is a one-dimensional Brownian motion.  The concavity method is not only applicable to second order deterministic  parabolic equations but also extended to a wide range of other types of deterministic as well as stochastic  evolution equations.

   Inspired by the above facts, in the present paper, our aim of this work is to study the blow-up phenomena by using the concavity method  for the following SPDE: 
   \begin{equation}\label{11}
	\left\{
   \begin{aligned}
	du(x,t)&=[\alpha\Delta u(x,t)+\beta|u(x,t)|^{m-1}u(x,t)]dt+\sigma(u(x,t),t)dW(t)\\&\quad+\int_{\mathrm{Z}}\eta(u(x,t),t,z)\widetilde{\pi}(dt,dz),\ x\in \mathcal{O},\ t>0,\\
	u(x,t)&=0, \ x\in\partial \mathcal{O}, t\geq 0,\\ 
	u(x,0)&=u_0(x),\ x\in \mathcal{O},
   \end{aligned}
\right.
\end{equation}
where $\mathcal{O}\subset\mathbb{R}^d$ be a bounded domain with smooth boundary $\partial\mathcal{O}$. Here, $\alpha,\beta>0$ and $m\geq 1$. Let $(\Omega,\mathscr{F},\{\mathscr{F}_t\}_{t\geq 0},\mathbb{P})$ be a filtered probability space satisfying the usual hypotheses, that is,  $\{\mathscr{F}_t\}_{t\geq 0}$ is complete and right continuous.  Let $\{W(t)\}_{t\geq 0}$ be a standard one-dimensional Brownian motion defined on the probability space $(\Omega,\mathscr{F},\{\mathscr{F}_t\}_{t\geq 0},\mathbb{P})$. Let $(\mathrm{Z},\mathcal{B}(\mathrm{Z}))$ be a locally compact Polish space. For any $B\in\mathcal{B}(\mathrm{Z})$ and $t\in[0,T]$, let  $\pi(t,B)$ be a Poisson random measure with intensity (L\'evy) measure $\lambda(B)$ and $\mathbb{E}[\pi(t,B)]=t\lambda(B)$. Moreover,   $\widetilde{\pi}(t,B)=\pi(t,B)-t\lambda(B)$ is the compensated Poisson random measure.  For example (\cite[Section 1.2.4]{DA}), let us take fix $\mathrm{Z}=\mathbb{R}^d$ and 	let $\lambda$ be a Borel measure defined on $\mathbb{R}^d$. We say that $\lambda$ is a L\'evy measure if  $\lambda(\{0\})=0$ and 
\begin{align}\label{LM}
	\int_{\mathbb{R}^d}\big(|y|^2\wedge 1\big)\lambda(d y)<\infty. 
\end{align}
Note that every L\'evy measure on $\mathbb{R}^d$ is $\sigma$-finite. In the rotationally invariant case,  the L\'evy measure can be defined in the following form:
\begin{align*}
	\lambda(d x)=\frac{c}{|x|^{d+\alpha}}d x, \ \text{ for  }\ \alpha\in(0,2),
\end{align*} for some constant $c$. Moreover, $\sigma(\cdot,\cdot)$ and $\eta(\cdot,\cdot,\cdot)$ denote the noise coefficients. In this work, we consider both additive and linear multiplicative noises. 

The Lebesgue spaces are denoted by $\mathrm{L}^{p}(\mathcal{O})$ for $p\in [1,\infty]$, and the norm in $\mathrm{L}^p(\mathcal{O})$ is denoted by $\|\cdot\|_{\mathrm{L}^p}$ and for $p=2$, the inner product in $\mathrm{L}^2(\mathcal{O})$ is represented by $(\cdot,\cdot)$.  Let $\C_{0}^{\infty}(\mathcal{O})$ denote the space of all infinite times differentiable functions having compact support in $\mathcal{O}$. We denote the Sobolev spaces by $\mathrm{H}^{k}(\mathcal{O}),$ for $k\in\mathbb{N}$. Let $\mathrm{H}_0^1(\mathcal{O})$ denote the closure of $\C_{0}^{\infty}(\mathcal{O})$ in the $\H^1$-norm $\|\cdot\|_{\mathrm{H}^1}:=\sqrt{\|\cdot\|_{\mathrm{L}^2}^2+\|\nabla\cdot\|_{\mathrm{L}^2}^2}$. As we are working in a bounded domain, by using the Poincar\'e inequality, we infer that the norm $\sqrt{\|\cdot\|_{\mathrm{L}^2}^2+\|\nabla\cdot\|_{\mathrm{L}^2}^2}$ is equivalent to the seminorm $\|\nabla\cdot\|_{\mathrm{L}^2}$ so that $\|\nabla\cdot\|_{\mathrm{L}^2}$ defines a norm on $\mathrm{H}_0^1(\mathcal{O})$. 

 In the additive noise case, if  $u_0\in\mathrm{L}^2(\Omega;\mathrm{H}_0^1(\mathcal{O}))$, $\sigma\in\mathrm{L}^2(\Omega;\mathrm{L}^{2}((0,T);\mathrm{H}^1(\mathcal{O})))$ and  $\eta\in\mathrm{L}^2(\Omega;\mathrm{L}^{2}((0,T)\times\mathrm{Z});\mathrm{H}^1(\mathcal{O}))$, that is,  if 
 \begin{align*}
& \mathbb{E}\left[\|u_0\|_{\mathrm{H}_0^1}^2\right]<\infty, \ \mathbb{E}\left[	\int_0^{T}[\|\sigma(t)\|_{\mathrm{L}^2}^2+\|\nabla \sigma(t)\|_{\mathrm{L}^2}^2]dt\right]<\infty, \\ &\mathbb{E}\left[\int_0^T\int_{\mathrm{Z}}[\|\eta(t,z)\|_{\mathrm{L}^2}^2+\|\nabla\eta(t,z)\|_{\mathrm{L}^2}^2]\lambda(dz)dt\right]<\infty,
 \end{align*}
for any $T>0$,   then by using the contraction mapping principle,  the local existence and uniqueness of a pathwise strong solution (in the probabilistic sense) $u\in\mathrm{L}^2(\Omega;\mathrm{L}^{\infty}(0,T^*;\mathrm{H}_0^1(\mathcal{O})),$ $0<T^*<T$,  to the problem \eqref{11} can be established  by using the same techniques as in  the proof  of \cite[Theorem 8.1]{PLC}. Moreover, $u$ has a modification with paths in $\mathrm{D}([0,T^*];\mathrm{H}_0^1(\mathcal{O}))$, $\mathbb{P}$-a.s., where $\mathrm{D}([0,T^*];\mathrm{H}_0^1(\mathcal{O}))$ denotes the space of all c\`adl\`ag  (right continuous  paths with left limits) process from $[0,T^*]$ to $\mathrm{H}_0^1(\mathcal{O})$.  A similar result holds true in the case of linear  multiplicative noise also. 
 
 The method we adopted in this work is motivated from \cite[Theorem 2.1]{GLJW}.   The concavity method in the deterministic case was suitably modified for the stochastic case by the authors in \cite{GLJW} to study  the blow-up problem of \eqref{1p2} perturbed by additive Gaussian noises.   But they were not able to obtain the results for the multiplicative noise case (see \cite[Remark 2.1]{GLJW}).  In this work, we are able to establish a blow-up criteria for the problem \eqref{11} perturbed by 
 \begin{itemize}
 	\item [(i)] additive L\'evy noise (Theorem \ref{thm2.1}),
 	\item [(ii)]  linear multiplicative  L\'evy noise (Theorem \ref{thm3.1}),
 \end{itemize}
by suitably modifying the method developed in \cite{GLJW} for the linear multiplicative noise case.  The main results of this work are given below.

\begin{theorem}[Additive noise case]\label{thm2.1}
	Let 
	\begin{align}\label{2p1}
		\sigma(u,t)=\sigma(x,t)\ \text{ and }\ \eta(u,t,z)=\eta(x,t,z)
	\end{align}
in \eqref{11}.	Assume that $m>1$,  $\alpha,\beta>0$. Let $\sigma(\cdot,t)$ and $\eta(\cdot,t,\cdot)$ be $\mathscr{F}_t$-adapted such that \begin{align*}\sigma&\in\mathrm{L}^2(\Omega;\mathrm{L}^{2}((0,\infty);\mathrm{L}^2(\mathcal{O})))\ \text{ with }\  \int_0^{\infty}\mathbb{E}\left[\|\nabla\sigma(t)\|_{\mathrm{L}^2}^2\right]dt<\infty,\nonumber\\
		\eta&\in\mathrm{L}^2(\Omega;\mathrm{L}^{2}((0,\infty)\times\mathrm{Z});\mathrm{L}^2(\mathcal{O}))\ \text{ with }\  \int_0^{\infty}\mathbb{E}\left[\int_{\mathrm{Z}}\|\nabla\eta(t,z)\|_{\mathrm{L}^2}^2\lambda(dz)\right]dt<\infty,\end{align*}    and the $\mathscr{F}_0$-measurable initial data $u_0$ satisfy 
	\begin{align}\label{12}
		-\frac{\alpha}{2}\mathbb{E}\left[\|\nabla u_0\|_{\mathrm{L}^2}^2\right]+\frac{\beta}{m+1}\mathbb{E}\left[\|u_0\|_{\mathrm{L}^{m+1}}^{m+1}\right]-\frac{\alpha}{2}\int_0^{\infty}\mathbb{E}\left[\|\nabla \sigma(s)\|_{\mathrm{L}^2}^2 +\int_{\mathrm{Z}}\|\nabla\eta(s,z)\|_{\mathrm{L}^2}^2\lambda(dz)\right]ds>0.
	\end{align}
	Then the solution of \eqref{11} with \eqref{2p1} as noise coefficient  must blow-up in a finite time in the sense of mean square, that is, there exists a time $0<T^*<T$ such that 
	\begin{align}
		\lim\limits_{t\uparrow T^*}\mathbb{E}\left[\|u(t)\|_{\mathrm{L}^2}^2\right]=\infty. 
	\end{align}
\end{theorem}

\begin{theorem}[Linear multiplicative noise case]\label{thm3.1} Let 
	\begin{align}\label{3p1}
		\sigma(u(x,t),t)=\sigma u(x,t)\ \text{ and }\ \eta(u(x,t),t,z)=\eta(z)u(x,t),
	\end{align}
in \eqref{11}, 	where $\sigma$ is a constant and $\eta(\cdot)\geq 0$ is independent of $t$. 
	Assume that $m>1$, $\alpha,\beta>0$ and fix 
	\begin{align}\label{3p2}
		\kappa= \frac{1}{2}\left(\sigma^2+\int_{\mathrm{Z}}\eta^2(z)\lambda(dz)\right)
	\end{align}
	such that 
	\begin{align}\label{3p3}
		0\leq\kappa\leq \alpha\lambda_1,
	\end{align}
	where $\lambda_1>0$ is the first eigenvalue of the Dirichlet Laplacian. Let  the $\mathscr{F}_0$-measurable initial data $u_0$ satisfy 
	\begin{align}\label{3p4}
		\mathbb{E}\left[-\frac{\alpha}{2}\|\nabla u_0\|_{\mathrm{L}^2}^2+\frac{\beta}{m+1}\|u_0\|_{\mathrm{L}^{m+1}}^{m+1}+\frac{\kappa}{m+1}\|u_0\|_{\mathrm{L}^2}^2\right]>0.
	\end{align}
	Then the solution of \eqref{11} with \eqref{3p1} as noise coefficient  must blow-up in a finite time in the sense of mean square.
\end{theorem}
We prove Theorem \ref{thm2.1} and Theorem \ref{thm3.1} in Section \ref{sec2} and Section \ref{sec3}, respectively. 


\section{Proof of Theorem \ref{thm2.1}}\label{sec2}\setcounter{equation}{0}
In this section, we prove Theorem \ref{thm2.1} by using the method developed in \cite{GLJW}.

\begin{proof}[Proof of Theorem \ref{thm2.1}]
	We prove the result through a contradiction argument. Let us assume that there exists a global solution $u$ such that
	\begin{equation}\label{2p3}
	\sup_{t\in[0,T]}\mathbb{E}\left[\|u(t)\|_{\mathrm{L}^2}^2\right]=	\sup_{t\in[0,T]}\mathbb{E}\left[\int_{\mathcal{O}}u^2(x,t)dx\right]<\infty, 
	\end{equation}
	for any $T>0$. Applying the infinite dimensional It\^o formula to the process $\|u(\cdot)\|_{\mathrm{L}^2}^2$ (\cite[Theorem 4.32]{GDJZ}, \cite[Proposition 1, Chapter 2]{MM}, \cite[Theorem B.2]{JZZB}), we obtain,  $\mathbb{P}$-a.s.,
	\begin{align}\label{13}
		\|u(t)\|_{\mathrm{L}^2}^2&=\|u_0\|_{\mathrm{L}^2}^2-2\alpha\int_0^t\|\nabla u(s)\|_{\mathrm{L}^2}^2ds+2\beta\int_0^t\|u(s)\|_{\mathrm{L}^{m+1}}^{m+1}ds\nonumber\\&\quad+2\int_0^t(\sigma(s),u(s))dW(s)+\int_0^t\|\sigma(s)\|_{\mathrm{L}^2}^2ds\nonumber\\&\quad+\int_0^t\int_{\mathrm{Z}}\left[\|u(s-)+\eta(s,z)\|_{\mathrm{L}^2}^2-\|u(s-)\|_{\mathrm{L}^2}^2\right]\widetilde{\pi}(ds,dz)\nonumber\\&\quad+\int_0^t\int_{\mathrm{Z}}\left[\|u(s)+\eta(s,z)\|_{\mathrm{L}^2}^2-\|u(s)\|_{\mathrm{L}^2}^2-2(\eta(s,z),u(s))\right]\lambda(dz)ds,
	\end{align}
for all $t\in[0,T]$. 	Taking expectation on both sides of \eqref{13} and then using the fact that the fourth and sixth terms appearing in the right hand side of \eqref{13} are martingales, we find for all $t\in[0,T]$,
	\begin{align}\label{14}
	\mathbb{E}\left[\|u(t)\|_{\mathrm{L}^2}^2\right]&=\mathbb{E}\left[\|u_0\|_{\mathrm{L}^2}^2\right]-2\alpha\int_0^t\mathbb{E}\left[\|\nabla u(s)\|_{\mathrm{L}^2}^2\right]ds+2\beta\int_0^t\mathbb{E}\left[\|u(s)\|_{\mathrm{L}^{m+1}}^{m+1}\right]ds\nonumber\\&\quad+\int_0^t\mathbb{E}\left[\|\sigma(s)\|_{\mathrm{L}^2}^2\right]ds+\int_0^t\int_{\mathrm{Z}}\mathbb{E}\left[\|\eta(s,z)\|_{\mathrm{L}^2}^2\right]\lambda(dz)ds. 
\end{align}
Let us set 
\begin{align*}
	v(t)&=\mathbb{E}\left[\|u(t)\|_{\mathrm{L}^2}^2\right], \\  h(t)&=\mathbb{E}\left[-2\alpha\|\nabla u(t)\|_{\mathrm{L}^2}^2+2\beta\|u(t)\|_{\mathrm{L}^{m+1}}^{m+1}+\|\sigma(t)\|_{\mathrm{L}^2}^2+\int_{\mathrm{Z}}\|\eta(t,z)\|_{\mathrm{L}^2}^2\lambda(dz)\right],
\end{align*}
so that \eqref{14} reduces to 
\begin{align*}
	v(t)-v(0)=\int_0^th(s)ds. 
\end{align*}
Next, we define 
\begin{align}\label{18}
	I(t)=\int_0^tv(s)ds+K, \ \mbox{ where $K$ is a positive constant.} 
\end{align}
Therefore, it is immediate that $I'(t)=v(t)$ and $I''(t)=v'(t)=h(t)$. Let us now set 
\begin{align}\label{19}
	J(t)=\mathbb{E}\left[-\frac{\alpha}{2}\|\nabla u(t)\|_{\mathrm{L}^2}^2+\frac{\beta}{m+1}\|u(t)\|_{\mathrm{L}^{m+1}}^{m+1}\right],
\end{align}
so that 
\begin{align}\label{110}
	I''(t)&=h(t)=\mathbb{E}\left[-2\alpha\|\nabla u(t)\|_{\mathrm{L}^2}^2+2\beta\|u(t)\|_{\mathrm{L}^{m+1}}^{m+1}+\|\sigma(t)\|_{\mathrm{L}^2}^2+\int_{\mathrm{Z}}\|\eta(t,z)\|_{\mathrm{L}^2}^2\lambda(dz)\right]\nonumber\\&\geq 2(m+1) \mathbb{E}\left[-\frac{\alpha}{m+1}\|\nabla u(s)\|_{\mathrm{L}^2}^2+\frac{\beta}{m+1}\|u(s)\|_{\mathrm{L}^{m+1}}^{m+1}\right]\nonumber\\&\geq 2(m+1)J(t). 
\end{align}
Applying the infinite dimensional It\^o formula to the process $\|\nabla u(\cdot)\|_{\mathrm{L}^2}^{2}$, we find, $\mathbb{P}$-a.s.,
\begin{align}\label{28}
	\|\nabla u(t)\|_{\mathrm{L}^2}^2&=\|\nabla u_0\|_{\mathrm{L}^2}^2-2\int_0^t(\Delta u(s),\alpha\Delta u(s)+\beta|u(s)|^{m-1}u(s))ds\nonumber\\&\quad+2\int_0^t(\nabla u(s),\nabla\sigma(s))dW(s)+\int_0^t\|\nabla \sigma(s)\|_{\mathrm{L}^2}^2ds \nonumber\\&\quad+\int_0^t\int_{\mathrm{Z}}\left[\|\nabla(u(s-)+\eta(s,z))\|_{\mathrm{L}^2}^2-\|\nabla u(s-)\|_{\mathrm{L}^2}^2\right]\widetilde{\pi}(ds,dz)\nonumber\\&\quad+\int_0^t\int_{\mathrm{Z}}\left[\|\nabla(u(s)+\eta(s,z))\|_{\mathrm{L}^2}^2-\|\nabla u(s)\|_{\mathrm{L}^2}^2-2(\nabla u(s),\nabla\eta(s,z))\right]\lambda(dz)ds,
\end{align}
for all $t\in[0,T]$. Taking expectation on both sides of the  inequality \eqref{28} and then using the fact that the third and fifth terms appearing in the right hand side of the equality \eqref{28}  are martingales, we deduce 
\begin{align}\label{112}
	\mathbb{E}\left[\|\nabla u(t)\|_{\mathrm{L}^2}^2\right]&=\mathbb{E}\left[\|\nabla u_0\|_{\mathrm{L}^2}^2\right]-2\mathbb{E}\left[\int_0^t(\Delta u(s),\alpha\Delta u(s)+\beta|u(s)|^{m-1}u(s))ds\right]\nonumber\\&\quad+\mathbb{E}\left[\int_0^t\|\nabla \sigma(s)\|_{\mathrm{L}^2}^2ds\right] +\mathbb{E}\left[\int_0^t\int_{\mathrm{Z}}\|\nabla\eta(s,z)\|_{\mathrm{L}^2}^2\lambda(dz)ds\right].
\end{align}
Once again applying  the infinite dimensional It\^o formula to the process $\|u(\cdot)\|_{\mathrm{L}^{m+1}}^{m+1}$, we find, $\mathbb{P}$-a.s.,
\begin{align}\label{113}
	&\|u(t)\|_{\mathrm{L}^{m+1}}^{m+1}\nonumber\\&=\|u_0\|_{\mathrm{L}^{m+1}}^{m+1}+(m+1)\int_0^t(|u(s)|^{m-1}u(s),\alpha\Delta u(s)+\beta|u(s)|^{m-1}u(s))ds\nonumber\\&\quad+(m+1)\int_0^t(|u(s)|^{m-1}u(s),\sigma(s))dW(s)+\frac{m(m+1)}{2}\int_0^t(|u(s)|^{m-1},|\sigma(s)|^2)ds\nonumber\\&\quad+\int_0^t\int_{\mathrm{Z}}\left[\|u(s-)+\eta(s)\|_{\mathrm{L}^{m+1}}^{m+1}-\|u(s-)\|_{\mathrm{L}^{m+1}}^{m+1}\right]\widetilde{\pi}(ds,dz)\nonumber\\&\quad+\int_0^t\int_{\mathrm{Z}}\left[\|u(s)+\eta(s,z)\|_{\mathrm{L}^{m+1}}^{m+1}-\|u(s)\|_{\mathrm{L}^{m+1}}^{m+1}-(m+1)(|u(s)|^{m-1}u(s),\eta(s,z))\right]\lambda(dz)ds, 
\end{align}
for all $t\in[0,T]$. Using Taylor's formula (see \cite[Theorem 7.9-1 (c)]{CI2013}), one can find  $0\leq \theta\leq 1$ such that 
\begin{align}\label{114}
	&\|u+\eta(\cdot)\|_{\mathrm{L}^{m+1}}^{m+1}-\|u\|_{\mathrm{L}^{m+1}}^{m+1}-(m+1)(|u|^{m-1}u,\eta(\cdot))\nonumber\\&= \frac{m(m+1)}{2}(|u+\theta \eta(\cdot)|^{m-1},|\eta(\cdot)|^2).
\end{align}
Using \eqref{114} the fact that the third and fifth terms appearing in the right hand side of the equality \eqref{113} are martingales, we further deduce 
\begin{align}\label{115}
	\mathbb{E}\left[\|u(t)\|_{\mathrm{L}^{m+1}}^{m+1}\right]&=\mathbb{E}\left[\|u_0\|_{\mathrm{L}^{m+1}}^{m+1}\right]+(m+1)\E\left[\int_0^t(|u(s)|^{m-1}u(s),\alpha\Delta u(s)+\beta|u(s)|^{m-1}u(s))ds\right]\nonumber\\&\quad+\frac{m(m+1)}{2}\mathbb{E}\left[\int_0^t(|u(s)|^{m-1},|\sigma(s)|^2)ds\right]\nonumber\\&\quad+\frac{m(m+1)}{2}\mathbb{E}\left[\int_0^t\int_{\mathrm{Z}}(|u(s)+\theta \eta(s,z)|^{m-1},|\eta(s,z)|^2)\lambda(dz)ds\right], 
\end{align}
for all $t\in[0,T]$. Combining \eqref{112} and \eqref{115} and substituting in \eqref{19}, we immediately have 
\begin{align}\label{216}
	J(t)&=J(0)+\mathbb{E}\left[\int_0^t\|\alpha\Delta u(s)+\beta|u(s)|^{m-1}u(s)\|_{\mathrm{L}^2}^2ds \right] \nonumber\\&\quad-\frac{\alpha}{2}\mathbb{E}\left[\int_0^t\left(\|\nabla \sigma(s)\|_{\mathrm{L}^2}^2 +\int_{\mathrm{Z}}\|\nabla\eta(s,z)\|_{\mathrm{L}^2}^2\lambda(dz)\right)ds\right] +\frac{m\beta}{2}\mathbb{E}\left[\int_0^t(|u(s)|^{m-1},|\sigma(s)|^2)ds\right]\nonumber\\&\quad+\frac{m\beta}{2}\mathbb{E}\left[\int_0^t\int_{\mathrm{Z}}(|u(s)+\theta \eta(s,z)|^{m-1},|\eta(s,z)|^2)\lambda(dz)ds\right], 
\end{align}
for all $t\in[0,T]$. Furthermore, we have  for all $t\in[0,T]$, 
\begin{align*}
	I'(t)&=v(t)
	=v(0)+\int_0^th(s)ds\nonumber\\&=v(0)+\int_0^t\mathbb{E}\left[\|\sigma(s)\|_{\mathrm{L}^2}^2+\int_{\mathrm{Z}}\|\eta(s,z)\|_{\mathrm{L}^2}^2\lambda(dz)\right]ds\nonumber\\&\quad+\int_0^t\mathbb{E}\left[-2\alpha\|\nabla u(s)\|_{\mathrm{L}^2}^2+2\beta\|u(s)\|_{\mathrm{L}^{m+1}}^{m+1}\right]ds\nonumber\\&= v(0)+\int_0^t\mathbb{E}\left[\|\sigma(s)\|_{\mathrm{L}^2}^2+\int_{\mathrm{Z}}\|\eta(s,z)\|_{\mathrm{L}^2}^2\lambda(dz)\right]ds\nonumber\\&\quad+2\int_0^t\mathbb{E}\left[\alpha(\Delta u(s),u(s))+\beta\|u(s)\|_{\mathrm{L}^{m+1}}^{m+1}\right]ds.
\end{align*}.
Therefore,  by using the Cauchy-Schwarz and H\"older's inequalities,  we deduce for any $\varepsilon>0$, 
\begin{align}\label{117}
	(I'(t))^2&=\left\{v(0)+\int_0^t\mathbb{E}\left[\|\sigma(s)\|_{\mathrm{L}^2}^2+\int_{\mathrm{Z}}\|\eta(s,z)\|_{\mathrm{L}^2}^2\lambda(dz)\right]ds\right\}^2\nonumber\\&\quad+4\left\{\int_0^t\mathbb{E}\left[\alpha(\Delta u(s),u(s))+\beta\|u(s)\|_{\mathrm{L}^{m+1}}^{m+1}\right]ds\right\}^2\nonumber\\&\quad+4\left\{v(0)+\int_0^t\mathbb{E}\left[\|\sigma(s)\|_{\mathrm{L}^2}^2+\int_{\mathrm{Z}}\|\eta(s,z)\|_{\mathrm{L}^2}^2\lambda(dz)\right]ds\right\}\nonumber\\&\quad\times\left\{\int_0^t\mathbb{E}\left[\alpha\Delta u(s)u(s)+\beta\|u(s)\|_{\mathrm{L}^{m+1}}^{m+1}\right]ds\right\}\nonumber\\&\leq 4(1+\varepsilon)\left\{\int_0^t\mathbb{E}\left[\alpha(\Delta u(s),u(s))+\beta(|u(s)|^{m-1}u(s),u(s))\right]ds\right\}^2\nonumber\\&\quad+\left(1+\frac{1}{\varepsilon}\right)\left\{v(0)+\int_0^t\mathbb{E}\left[\|\sigma(s)\|_{\mathrm{L}^2}^2+\int_{\mathrm{Z}}\|\eta(s,z)\|_{\mathrm{L}^2}^2\lambda(dz)\right]ds\right\}^2\nonumber\\&\leq 4(1+\varepsilon)\left\{\int_0^t\mathbb{E}\left[\|\alpha\Delta u(s)+\beta|u(s)|^{m-1}u(s)\|_{\mathrm{L}^2}^2\right]ds\right\}\left\{\int_0^t\mathbb{E}\left[\|u(s)\|_{\mathrm{L}^2}^2\right]ds\right\}\nonumber\\&\quad+\left(1+\frac{1}{\varepsilon}\right)\left\{v(0)+\int_0^t\mathbb{E}\left[\|\sigma(s)\|_{\mathrm{L}^2}^2+\int_{\mathrm{Z}}\|\eta(s,z)\|_{\mathrm{L}^2}^2\lambda(dz)\right]ds\right\}^2. 
\end{align}
Combining \eqref{18}, \eqref{110} and \eqref{117}, we obtain 
\begin{align}\label{118}
&	I''(t)I(t)-(1+\delta)(I'(t))^2\nonumber\\&\geq 2(m+1)J(t)-(1+\delta)(I'(t))^2\nonumber\\&\geq 2(m+1)\bigg\{J(0)+\int_0^t\mathbb{E}\left[\|\alpha\Delta u(s)+\beta|u(s)|^{m-1}u(s)\|_{\mathrm{L}^2}^2\right] ds \nonumber\\&\quad-\frac{\alpha}{2}\int_0^t\mathbb{E}\left[\|\nabla \sigma(s)\|_{\mathrm{L}^2}^2 +\int_{\mathrm{Z}}\|\nabla\eta(s,z)\|_{\mathrm{L}^2}^2\lambda(dz)\right]ds +\frac{m\beta}{2}\int_0^t\mathbb{E}\left[(|u(s)|^{m-1},|\sigma(s)|^2)\right]ds\nonumber\\&\quad+\frac{m\beta}{2}\int_0^t\mathbb{E}\left[\int_{\mathrm{Z}}(|u(s)+\theta \eta(s,z)|^{m-1},|\eta(s,z)|^2)\lambda(dz)\right]ds\bigg\}\nonumber\\&\qquad\times\left\{\int_0^t\mathbb{E}\left[\|u(s)\|_{\mathrm{L}^2}^2\right]ds+K\right\}\nonumber\\&\quad -4(1+\varepsilon)(1+\delta)\left\{\int_0^t\mathbb{E}\left[\|\alpha\Delta u(s)+\beta|u(s)|^{m-1}u(s)\|_{\mathrm{L}^2}^2\right]ds\right\}\left\{\int_0^t\mathbb{E}\left[\|u(s)\|_{\mathrm{L}^2}^2\right]ds\right\}\nonumber\\&\quad-\left(1+\frac{1}{\varepsilon}\right)(1+\delta)\left\{v(0)+\int_0^t\mathbb{E}\left[\|\sigma(s)\|_{\mathrm{L}^2}^2+\int_{\mathrm{Z}}\|\eta(s,z)\|_{\mathrm{L}^2}^2\lambda(dz)\right]ds\right\}^2\nonumber\\&\geq \left[2(m+1)-4(1+\varepsilon)(1+\delta)\right]\nonumber\\&\qquad\times\left\{\int_0^t\mathbb{E}\left[\|\alpha\Delta u(s)+\beta|u(s)|^{m-1}u(s)\|_{\mathrm{L}^2}^2\right]ds\right\}\left\{\int_0^t\mathbb{E}\left[\|u(s)\|_{\mathrm{L}^2}^2\right]ds\right\}\nonumber\\&\quad+2(1+m)K\left\{J(0)-\frac{\alpha}{2}\int_0^t\mathbb{E}\left[\|\nabla \sigma(s)\|_{\mathrm{L}^2}^2 +\int_{\mathrm{Z}}\|\nabla\eta(s,z)\|_{\mathrm{L}^2}^2\lambda(dz)\right]ds\right\}\nonumber\\&\quad-\left(1+\frac{1}{\varepsilon}\right)(1+\delta)\left\{v(0)+\int_0^t\mathbb{E}\left[\|\sigma(s)\|_{\mathrm{L}^2}^2+\int_{\mathrm{Z}}\|\eta(s,z)\|_{\mathrm{L}^2}^2\lambda(dz)\right]ds\right\}^2,
	\end{align}
	where the constant $\delta$ will be fixed later. 	Let us choose $\varepsilon=\frac{m-1}{4}$ and $\delta=\frac{m-1}{2(m+3)}$, so that 
	\begin{align}\label{119}
		2(m+1)-4(1+\varepsilon)(1+\delta)=\frac{m-1}{2}>0,
	\end{align}
	for $m>1$. On the other hand, the condition \eqref{12} implies 
	\begin{align}\label{120}
		J(0)-\frac{\alpha}{2}\int_0^t\mathbb{E}\left[\|\nabla \sigma(s)\|_{\mathrm{L}^2}^2 +\int_{\mathrm{Z}}\|\nabla\eta(s,z)\|_{\mathrm{L}^2}^2\lambda(dz)\right]ds>0.
	\end{align}
	Combining \eqref{119} and \eqref{120}, substituting it in \eqref{118} and then choosing $K$ large enough, one can conclude that 
	\begin{align*}
			I''(t)I(t)-(1+\delta)(I'(t))^2>0\Rightarrow \frac{d}{dt}\left(\frac{I'(t)}{I^{1+\delta}(t)}\right)>0\Rightarrow \frac{I'(t)}{I^{1+\delta}(t)}>\frac{I'(0)}{I^{1+\delta}(0)},
	\end{align*}
	for any $t>0$. It follows that $I(t)$ cannot remain finite for all $t$, which leads to a contradiction, and the proof is complete. 
\end{proof}

\section{Proof of Theorem \ref{thm3.1}}\label{sec3}\setcounter{equation}{0}
In this section,  we prove Theorem \ref{thm3.1} by suitably modifying the method developed in \cite{GLJW}.  

\begin{proof}[Proof of Theorem \ref{thm3.1}]
We prove this theorem also by contradiction. Let us assume that \eqref{2p3} holds true for any $T>0$. 	For $\kappa$ defined in \eqref{3p2}, calculations similar to \eqref{14}, \eqref{112} and \eqref{115} yield
	\begin{align}
		\mathbb{E}\left[\|u(t)\|_{\mathrm{L}^2}^2\right]&=\mathbb{E}\left[\|u_0\|_{\mathrm{L}^2}^2\right]-2\alpha\int_0^t\mathbb{E}\left[\|\nabla u(s)\|_{\mathrm{L}^2}^2\right]ds+2\beta\int_0^t\mathbb{E}\left[\|u(s)\|_{\mathrm{L}^{m+1}}^{m+1}\right]ds\nonumber\\&\quad+2\kappa\int_0^t\mathbb{E}\left[\|u(s)\|_{\mathrm{L}^2}^2\right]ds,\label{31}\\&=\mathbb{E}\left[\|u_0\|_{\mathrm{L}^2}^2\right]+2\int_0^t\mathbb{E}\left[(\alpha\Delta u(s)+\beta|u(s)|^{m-1}u(s)+\kappa u(s),u(s))\right]ds,\\
			\mathbb{E}\left[\|\nabla u(t)\|_{\mathrm{L}^2}^2\right]&=\mathbb{E}\left[\|\nabla u_0\|_{\mathrm{L}^2}^2\right]-2\mathbb{E}\left[\int_0^t(\Delta u(s),\alpha\Delta u(s)+\beta|u(s)|^{m-1}u(s))ds\right]\nonumber\\&\quad+2\kappa\mathbb{E}\left[\int_0^t\|\nabla u(s)\|_{\mathrm{L}^2}^2ds\right]\nonumber\\
			&=\mathbb{E}\left[\|\nabla u_0\|_{\mathrm{L}^2}^2\right]-2\mathbb{E}\left[\int_0^t(\Delta u(s),\alpha\Delta u(s)+\beta|u(s)|^{m-1}u(s)+\kappa u(s))ds\right],\\
			\mathbb{E}\left[\|u(t)\|_{\mathrm{L}^{m+1}}^{m+1}\right]&=\mathbb{E}\left[\|u_0\|_{\mathrm{L}^{m+1}}^{m+1}\right]+(m+1)\E\left[\int_0^t(|u(s)|^{m-1}u(s),\alpha\Delta u(s)+\beta|u(s)|^{m-1}u(s))ds\right]\nonumber\\&\quad+\frac{m(m+1)}{2}\sigma^2\mathbb{E}\left[\int_0^t\|u(s)\|_{\mathrm{L}^{m+1}}^{m+1}ds\right]\nonumber\\&\quad+\frac{m(m+1)}{2}\mathbb{E}\left[\int_0^t\|u(s)\|_{\mathrm{L}^{m+1}}^{m+1}\int_{\mathrm{Z}}|1+\theta \eta(z)|^{m-1}|\eta(z)|^2\lambda(dz)ds\right]\nonumber\\&\geq \mathbb{E}\left[\|u_0\|_{\mathrm{L}^{m+1}}^{m+1}\right]+(m+1)\E\left[\int_0^t(|u(s)|^{m-1}u(s),\alpha\Delta u(s)+\beta|u(s)|^{m-1}u(s))ds\right]\nonumber\\&\quad+{m(m+1)}\kappa\mathbb{E}\left[\int_0^t\|u(s)\|_{\mathrm{L}^{m+1}}^{m+1}ds\right] \nonumber\\&=\mathbb{E}\left[\|u_0\|_{\mathrm{L}^{m+1}}^{m+1}\right]+(m+1)\E\bigg[\int_0^t(|u(s)|^{m-1}u(s),\alpha\Delta u(s)+\beta|u(s)|^{m-1}u(s)\nonumber\\&\qquad+\kappa u(s))ds\bigg]+{(m+1)(m-1)}\kappa\mathbb{E}\left[\int_0^t\|u(s)\|_{\mathrm{L}^{m+1}}^{m+1}ds\right] ,
	\end{align}
	for all $t\in[0,T]$. As in the previous theorem, let us set 
	\begin{align*}
		v(t)&=\mathbb{E}\left[\|u(t)\|_{\mathrm{L}^2}^2\right], \\  h(t)&=\mathbb{E}\left[-2\alpha\|\nabla u(t)\|_{\mathrm{L}^2}^2+2\beta\|u(t)\|_{\mathrm{L}^{m+1}}^{m+1}+2\kappa\|u(t)\|_{\mathrm{L}^2}^2\right],
	\end{align*}
	so that \eqref{31} reduces to 
	\begin{align*}
		v(t)-v(0)=\int_0^th(s)ds. 
	\end{align*}
	Next, we define 
	\begin{align}\label{38}
		I(t)=\int_0^tv(s)ds+K, \ \mbox{ where $K$ is a positive constant.} 
	\end{align}
	Therefore, it is immediate that $I'(t)=v(t)$ and $I''(t)=v'(t)=h(t)$.

	Let us now set 
	\begin{align}\label{34}
		\widetilde{J}(t)= \mathbb{E}\left[-\frac{\alpha}{2}\|\nabla u(t)\|_{\mathrm{L}^2}^2+\frac{\beta}{m+1}\|u(t)\|_{\mathrm{L}^{m+1}}^{m+1}+\frac{\kappa}{m+1}\|u(t)\|_{\mathrm{L}^2}^2\right].
	\end{align}
Moreover, we have 
	\begin{align}
		I''(t)&=h(t)=\mathbb{E}\left[-2\alpha\|\nabla u(t)\|_{\mathrm{L}^2}^2+2\beta\|u(t)\|_{\mathrm{L}^{m+1}}^{m+1}+2\kappa\|u(t)\|_{\mathrm{L}^2}^2\right]\nonumber\\&\geq 2(m+1) \mathbb{E}\left[-\frac{\alpha}{m+1}\|\nabla u(s)\|_{\mathrm{L}^2}^2+\frac{\beta}{m+1}\|u(s)\|_{\mathrm{L}^{m+1}}^{m+1}+\frac{\kappa}{m+1}\|u(t)\|_{\mathrm{L}^2}^2\right]\nonumber\\&\geq 2(m+1)\widetilde{J}(t). 
	\end{align}
	A calculation similar to \eqref{216} yields 
	\begin{align}\label{311}
		\widetilde{J}(t)&\geq\widetilde{J}(0)+\int_0^t\mathbb{E}\left[\|\alpha\Delta u(s)+\beta|u(s)|^{m-1}u(s)+\kappa u(s)\|_{\mathrm{L}^2}^2\right] ds \nonumber\\&\quad+\frac{m-1}{m+1}\kappa\alpha\int_0^t\mathbb{E}\left[\|\nabla u(s)\|_{\mathrm{L}^2}^2\right]ds+\left((m-1)\kappa\beta-\frac{m-1}{m+1}\kappa\beta\right)\int_0^t\mathbb{E}\left[\| u(s)\|_{\mathrm{L}^{m+1}}^{m+1}\right]ds\nonumber\\&\quad-\frac{m-1}{m+1}\kappa^2\int_0^t\mathbb{E}\left[\|u(s)\|_{\mathrm{L}^2}^2\right]ds\nonumber\\&\geq \widetilde{J}(0)+\int_0^t\mathbb{E}\left[\|\alpha\Delta u(s)+\beta|u(s)|^{m-1}u(s)+\kappa u(s)\|_{\mathrm{L}^2}^2\right] ds \nonumber\\&\quad+\frac{m-1}{m+1}\kappa\left(\alpha-\frac{\kappa}{\lambda_1}\right)\int_0^t\mathbb{E}\left[\|\nabla u(s)\|_{\mathrm{L}^2}^2\right]ds+\frac{m(m-1)}{m+1}\kappa\beta\int_0^t\mathbb{E}\left[\| u(s)\|_{\mathrm{L}^{m+1}}^{m+1}\right]ds,
		\end{align}
		where  $\lambda_1$ is the first eigenvalue of the Dirichlet Laplacian. Note that 
		\begin{align*}
			I'(t)&=v(0)+2\int_0^t\mathbb{E}\left[\alpha(\Delta u(s),u(s))+\beta\|u(s)\|_{\mathrm{L}^{m+1}}^{m+1}\right]ds+2\kappa\int_0^t\mathbb{E}\left[\|u(s)\|_{\mathrm{L}^2}^2\right]ds,
		\end{align*}
		and 
		\begin{align}\label{312}
			(I'(t))^2&=v^2(0)+4\left\{\int_0^t\mathbb{E}\left[(\alpha\Delta u(s)+\beta|u(s)|^{m-1}u(s)+\kappa u(s),u(s))\right]ds\right\}^2\nonumber\\&\quad+4v(0)\left\{\int_0^t\mathbb{E}\left[(\alpha\Delta u(s)+\beta|u(s)|^{m-1}u(s)+\kappa u(s),u(s))\right]ds\right\}\nonumber\\&\leq 4(1+\varepsilon)\left\{\int_0^t\mathbb{E}\left[\|\alpha\Delta u(s)+\beta|u(s)|^{m-1}u(s)+\kappa u(s)\|_{\mathrm{L}^2}^2\right]ds\right\}\left\{\int_0^t\mathbb{E}\left[\|u(s)\|_{\mathrm{L}^2}^2\right]ds\right\}\nonumber\\&\quad+\left(1+\frac{1}{\varepsilon}\right)v^2(0). 
		\end{align}
		Combining \eqref{38}, \eqref{311} and \eqref{312},  for $0\leq \kappa \leq\alpha\lambda_1$ (see \eqref{3p3}) and $m>1$, we obtain
		\begin{align}\label{313}
			&	I''(t)I(t)-(1+\delta)(I'(t))^2\nonumber\\&\geq 2(m+1)\widetilde{J}(t)-(1+\delta)(I'(t))^2\nonumber\\&\geq 2(m+1)\bigg\{\widetilde{J}(0)+\int_0^t\mathbb{E}\left[\|\alpha\Delta u(s)+\beta|u(s)|^{m-1}u(s)+\kappa u(s)\|_{\mathrm{L}^2}^2\right] ds \nonumber\\&\quad+\frac{m-1}{m+1}\kappa\left(\alpha-\frac{\kappa}{\lambda_1}\right)\int_0^t\mathbb{E}\left[\|\nabla u(s)\|_{\mathrm{L}^2}^2\right]ds+\frac{m(m-1)}{m+1}\kappa\beta\int_0^t\mathbb{E}\left[\| u(s)\|_{\mathrm{L}^{m+1}}^{m+1}\right]ds\bigg\}\nonumber\\&\qquad\times\left\{\int_0^t\mathbb{E}\left[\|u(s)\|_{\mathrm{L}^2}^2\right]ds+K\right\}\nonumber\\&\quad -4(1+\varepsilon)(1+\delta)\left\{\int_0^t\mathbb{E}\left[\|\alpha\Delta u(s)+\beta|u(s)|^{m-1}u(s)+\kappa u(s)\|_{\mathrm{L}^2}^2\right]ds\right\}\left\{\int_0^t\mathbb{E}\left[\|u(s)\|_{\mathrm{L}^2}^2\right]ds\right\}\nonumber\\&\quad-\left(1+\frac{1}{\varepsilon}\right)(1+\delta)v^2(0)\nonumber\\&\geq \left[2(m+1)-4(1+\varepsilon)(1+\delta)\right]\nonumber\\&\qquad\times\left\{\int_0^t\mathbb{E}\left[\|\alpha\Delta u(s)+\beta|u(s)|^{m-1}u(s)+\kappa u(s)\|_{\mathrm{L}^2}^2\right]ds\right\}\left\{\int_0^t\mathbb{E}\left[\|u(s)\|_{\mathrm{L}^2}^2\right]ds\right\}\nonumber\\&\quad +2(m+1)K\widetilde{J}(0)-\left(1+\frac{1}{\varepsilon}\right)(1+\delta)v^2(0),
		\end{align}
		where the constant $\delta$ will be chosen later. Let us fix $\varepsilon=\frac{m-1}{4}$ and $\delta=\frac{m-1}{2(m+3)}$, so that \eqref{119} is satisfied 	for $m>1$. On the other hand, the condition \eqref{3p4} implies $\widetilde{J}(0)>0$. For sufficiently large $K$, one can complete the proof by following the proof of Theorem \ref{thm2.1}. 
\end{proof}

\begin{remark}
	As pointed out in \cite[Remark 2.1]{GLJW}, the main advantage of concavity method is that the positivity of solutions is not used in Theorems \ref{thm2.1} and \ref{thm3.1}. 
\end{remark}

	\medskip\noindent
	{\bf Acknowledgments:} M. T. Mohan would  like to thank the Department of Science and Technology (DST) Science $\&$ Engineering Research Board (SERB), India for a MATRICS grant (MTR/2021/000066). S. Sankar expresses gratitude to Periyar University for providing the University Research Fellowship. This study was completed through a discussion that took place during the visit of Dr. Manil T. Mohan to Periyar University under the Visiting Professor Scheme.

\end{document}